\documentclass[oneside]{amsart}
\usepackage[T1]{fontenc}
\usepackage[latin9]{inputenc}
\usepackage{amsthm}
\usepackage{amstext}
\usepackage{amssymb}

\makeatletter
\numberwithin{equation}{section}
\numberwithin{figure}{section}
\theoremstyle{plain}
\newtheorem{thm}{\protect\theoremname}[section]
  \theoremstyle{definition}
  \newtheorem{defn}[thm]{\protect\definitionname}
  \theoremstyle{plain}
  \newtheorem{prop}[thm]{\protect\propositionname}
  \theoremstyle{plain}
  \newtheorem{lem}[thm]{\protect\lemmaname}
  \theoremstyle{remark}
  \newtheorem*{rem*}{\protect\remarkname}
  \theoremstyle{plain}
  \newtheorem{cor}[thm]{\protect\corollaryname}
  \theoremstyle{remark}
  \newtheorem*{acknowledgement*}{\protect\acknowledgementname}

\DeclareMathOperator{\ad}{ad}
\DeclareMathOperator{\im}{Im}

\DeclareMathOperator{\Hom}{Hom}

\DeclareMathOperator{\Aut}{Aut}

\DeclareMathOperator{\T}{T}
\DeclareMathOperator{\N}{N}
\DeclareMathOperator{\C}{C}

\DeclareMathOperator{\Spec}{Spec}

\DeclareMathOperator{\chara}{char} 

\newcommand{\iso}{\mathbin{\kern.15em\widetilde{\hphantom{\hspace{.6em}}}\kern-.98em\rightarrow\kern.05em}}
\newcommand{\longiso}{\mathbin{\kern.3em\widetilde{\hphantom{\hspace{1.1em}}}\kern-1.55em\longrightarrow\kern.1em}}

\usepackage[all]{xy}
\SelectTips{cm}{10}

\makeatletter
\newcommand{\xyC}[1]{%
\makeatletter
\xydef@\xymatrixcolsep@{#1}
\makeatother
} 

\theoremstyle{plain}
\newtheorem*{thm1}{Theorem \ref{thm:Cartan}}

\newtheorem*{cor1}{Corollary \ref{cor:F - Norm}}

\date{}

\makeatother

  \providecommand{\acknowledgementname}{Acknowledgement}
  \providecommand{\corollaryname}{Corollary}
  \providecommand{\definitionname}{Definition}
  \providecommand{\lemmaname}{Lemma}
  \providecommand{\propositionname}{Proposition}
  \providecommand{\remarkname}{Remark}
\providecommand{\theoremname}{Theorem}

\begin{document}

\title{Reductive group schemes, the Greenberg functor, and associated algebraic
groups}

\author{Alexander Stasinski}

\address{School of Mathematics\\
University of Southampton\\
Southampton, SO17 1BJ\\
UK}

\email{\texttt{a.stasinski@soton.ac.uk}}
\begin{abstract}
Let $A$ be an Artinian local ring with algebraically closed residue
field $k$, and let $\mathbf{G}$ be an affine smooth group scheme
over $A$. The Greenberg functor $\mathcal{F}$ associates to $\mathbf{G}$
a linear algebraic group $G:=(\mathcal{F}\mathbf{G})(k)$ over $k$,
such that $G\cong\mathbf{G}(A)$. We prove that if $\mathbf{G}$ is
a reductive group scheme over $A$, and $\mathbf{T}$ is a maximal
torus of $\mathbf{G}$, then $T$ is a Cartan subgroup of $G$, and
every Cartan subgroup of $G$ is obtained uniquely in this way. Moreover,
we prove that if $\mathbf{G}$ is reductive and $\mathbf{P}$ is a
parabolic subgroup of $\mathbf{G}$, then $P$ is a self-normalising
subgroup of $G$, and if $\mathbf{B}$ and $\mathbf{B}'$ are two
Borel subgroups of $\mathbf{G}$, then the corresponding subgroups
$B$ and $B'$ are conjugate in $G$.
\end{abstract}
\maketitle

\section*{Errata notes}

A previous version of the present paper has appeared in J.~Pure Appl.~Algebra.,
\textbf{216} (2012), 1092-1101. After the publication the author was
notified by Cristian D. Gonz\'alez-Avil\'es and Alessandra Bertapelle
that the formula on p.~1094, l.~14 in the published version does
not hold in general. More precisely, in \cite{greenberg1}, p.~636
Greenberg defined the local ring scheme $\mathbf{A}$ over $k$ (using
different notation). However, the formula 
\[
\mathbf{A}(R)=A\otimes_{W_{m}(k)}W_{m}(R)
\]
does not hold for all $k$-algebras $R$. Nevertheless, it does hold
when $R$ is a \emph{perfect }$k$-algebra. The formula occurs in
\cite{Neron_models}, p.~276, l.~-18, but was stated correctly by
Loeser and Sebag \cite{Loeser-Sebag}, p.~318. Moreover, Nicaise
and Sebag have given counter-examples for non-perfect algebras; see
\cite{Nicaise-Sebag}, 2.2.

In the present paper we have corrected the statements involving the
above formula by either removing the formula or adding the hypothesis
that $R$ is a perfect $k$-algebra. The corrections correspond to
p.~1094, l.~14 and l.~-17, as well as the second line of the proof
of Lemma~2.3, in the published version. The formula is actually not
necessary for the results of our paper, and can simply be ignored.
All that is needed is that $\mathbf{A}$ is an affine local ring scheme
over $k$, and that it is isomorphic to some affine space. As mentioned
above, these facts were established by Greenberg.

Apart from these modifications, the content of the present version
remains identical to the published version.

\section{Introduction}

Ever since the work of Steinberg \cite{Steinberg-End} and Deligne
and Lusztig \cite{delignelusztig}, it has been known that the structure
of connected reductive algebraic groups plays an important role in
the representation theory of finite groups of Lie type (i.e., reductive
groups over finite fields). More recently, generalisations of the
construction of Deligne and Lusztig to reductive groups over finite
local rings have appeared in \cite{Lusztig-Fin-Rings,Alex_Unramified_reps,EDLs}.
In these generalisations, the role of the connected reductive groups
is taken over by certain connected (non-reductive) algebraic groups
associated to reductive group schemes over Artinian local rings, via
the Greenberg functor. In this paper we develop some of the structure
theory of these algebraic groups. These results allow for a smoother
treatment of parts of the construction in \cite{Alex_Unramified_reps},
and are necessary (but not sufficient) for a generalisation of the
construction in \cite{EDLs} beyond general and special linear groups.
The algebraic groups we consider are extensions of reductive groups
by connected unipotent groups, and as such are generally not reductive.
Nevertheless, we show that they possess subgroups with properties
closely analogous to subgroups in reductive algebraic groups. 

Let $\text{Set}$ denote the category of sets, and $\text{CRing}$
the category of commutative associative unital rings. Throughout this
paper, a ring will always refer to an object in $\text{CRing}$. As
usual, we will speak of a scheme $\mathbf{X}$ over a ring $R$ rather
than over $\Spec R$, and we write $\mathbf{X}(R)$ for the points
of $\mathbf{X}$ in $\Spec R$. Let $A$ be an Artinian local ring
with algebraically closed residue field $k$. Let $\mathbf{X}$ be
a scheme locally of finite type over $A$. Greenberg \cite{greenberg1}
has defined a functor $\mathcal{F}:\mathbf{X}\mapsto\mathcal{F}\mathbf{X}$
from the category of schemes locally of finite type over $A$ to the
category of schemes locally of finite type over $k$, with the property
that there is a canonical bijection
\[
\mathbf{X}(A)\longiso(\mathcal{F}\mathbf{X})(k).
\]
In Section~\ref{sec:Greenberg functor} we view schemes in terms
of their functors of points, and define the Greenberg functor more
generally for any functor $\text{CRing}\rightarrow\text{Set}$. The
functor $\mathcal{F}$ enjoys a number of properties, proved in \cite{greenberg1,greenberg2}:
If $\mathbf{X}$ is a group scheme over $A$, then $\mathcal{F}\mathbf{X}$
is naturally a group scheme over $k$, and the above bijection is
a group isomorphism. If $\mathbf{X}$ is affine or smooth over $A$,
then the same is true for $\mathcal{F}\mathbf{X}$ over $k$, respectively.
If $\mathbf{X}$ is smooth over $A$ and $\mathbf{X}\times k$ is
irreducible, then $\mathcal{F}\mathbf{X}$ is irreducible (see \cite{greenberg2},
p.~264, Corollary~2; note that since $k$ is algebraically closed
a smooth scheme over $k$ is automatically reduced). Furthermore,
$\mathcal{F}$ preserves open and closed subschemes, respectively.

Let $\mathbf{G}$ be an affine smooth group scheme over $A$. Then
it is in particular of finite type over $A$. Define the group
\[
G=(\mathcal{F}\mathbf{G})(k).
\]
By Greenberg's results mentioned above, $G$ is the $k$-points of
an affine smooth group scheme over $k$, that is, $G$ is a linear
algebraic group over $k$. In general, we write group schemes over
$A$ in boldface type, and the corresponding algebraic group over
$k$ associated to the group scheme via the Greenberg functor as above,
using the same letter in normal type. The group $G$ is connected
if its fibre $\mathbf{G}\times k$ is. 

Suppose that $\mathbf{G}$ is a reductive group scheme over $A$,
that is, an affine smooth group scheme over $A$, such that its fibre
$\mathbf{G}\times k$ is a connected reductive group over $k$ in
the classical sense. Let $\mathbf{H}$ be a subscheme of $\mathbf{G}$.
One can define the normaliser group functor $\N_{\mathbf{G}}(\mathbf{H})$
(see Section~\ref{sec:Transp-norm}) which, as we will see, is often
representable by a closed subscheme of $\mathbf{G}$. Let $\mathbf{T}$
be a maximal torus of $\mathbf{G}$ (see \cite{SGA3}, XII~1.3 and
XV~6.1). Then $\mathbf{T}$ is affine smooth over $A$, and its fibre
$\mathbf{T}\times k$ is a maximal torus of $\mathbf{G}\times k$
in the classical sense. Recall that a Cartan subgroup of a linear
algebraic group over $k$ is defined as the centraliser of a maximal
torus (see \cite{borel}, 11.13). When $A$ is not a field, the group
$G$ is no longer reductive, and the subgroups of the form $T$ are
not maximal tori. We will however prove the following:\begin{thm1}Let
$\mathbf{G}$ be a reductive group scheme over $A$, and let $\mathbf{T}$
be a maximal torus in $\mathbf{G}$. Then $T$ is a Cartan subgroup
of $G$, and the map $\mathbf{T}\mapsto T$ is a bijection between
the set of maximal tori in $\mathbf{G}$ and the set of Cartan subgroups
of $G$.\end{thm1}In particular, it follows from this result that
the groups $T$ are all conjugate in $G$. The groups $G$ thus form
a large family of connected linear algebraic groups with non-trivial
maximal tori and abelian Cartan subgroups which are generally not
maximal tori. 

The proof of the above theorem (and other results of this paper) is
based on the following observation. Recall that Hilbert's Nullstellensatz
implies that if $\mathbf{X}$ is an affine variety over an algebraically
closed field $k$, and $\mathbf{Y}$ and $\mathbf{Z}$ are two closed
reduced subvarieties of $\mathbf{X}$, then $\mathbf{Y}=\mathbf{Y}'$
(i.e., $\mathbf{Y}$ and $\mathbf{Y}'$ are isomorphic as subvarieties
of $\mathbf{X}$) if and only if $\mathbf{Y}(k)=\mathbf{Y}'(k)$,
as subsets of $\mathbf{X}(k)$. In certain situations, this result
can be ``lifted'' to schemes over $A$. More precisely, in Proposition~\ref{pro:Nullstellensatz}
we show that if $\mathbf{X}$ is an affine scheme of finite type over
$A$, and $\mathbf{Y}$ and $\mathbf{Y}'$ are closed smooth subschemes
of $\mathbf{X}$, then $\mathbf{Y}=\mathbf{Y}'$ if and only if $\mathbf{Y}(A)=\mathbf{Y}'(A)$.
As a consequence of this we prove that the Greenberg functor is, in
a certain sense, compatible with the formation of certain normaliser
group schemes, or more generally, transporters, over $A$ and $k$,
respectively (see Proposition~\ref{pro:F and Tran} for the precise
statement). An important special case of this is\begin{cor1}Let $\mathbf{G}$
be an affine group scheme of finite type over $A$, and let $\mathbf{H}$
be a closed smooth subgroup scheme. Assume that $\N_{\mathbf{G}}(\mathbf{H})$
is representable by a closed smooth subscheme of $\mathbf{G}$. Then
\[
\mathcal{F}\N_{\mathbf{G}}(\mathbf{H})(k)=\N_{\mathcal{F}\mathbf{G}}(\mathcal{F}\mathbf{H})(k).
\]
\end{cor1} It is this result, together with the fact that the Greenberg
functor preserves connected components of smooth group schemes over
$A$, which is the key to our proof of Theorem~\ref{thm:Cartan}.

Another type of group which plays an important role in the structure
theory of reductive groups are parabolic subgroups and Borel subgroups.
Given a Borel subgroup $\mathbf{B}$ of $\mathbf{G}$, the corresponding
subgroup $B$ of $G$ is not in general a Borel subgroup. However,
the constructions in \cite{Lusztig-Fin-Rings,Alex_Unramified_reps,EDLs}
show that the groups $B$ play the role of Borel subgroups in the
generalised Deligne-Lusztig theory. In \cite{EDLs}, groups of the
form $B$ are called \emph{strict Borel subgroups}. To have a useful
analogy between strict Borels and Borel subgroups of reductive groups,
it is important to establish that strict Borels are self-normalising
in $G$, and that they form a single orbit under conjugation in $G$.
In Proposition~\ref{pro:N(P)-B} we prove these facts using Proposition~\ref{pro:F and Tran}
together with some results from SGA~3 on smoothness of transporters.

\section{\label{sec:Greenberg functor}Functors of points and the Greenberg
functor}

We will follow the common practice of ignoring set-theoretical complications
in our use of categories. The appropriate modifications can be achieved
for example by using universes, as in \cite{Demazure-Gabriel} (see
also the English translation of its first two chapters \cite{Demazure-Gabriel-English}).
When dealing with group schemes, it is convenient to take the ``functor
of points'' point of view. We therefore begin this section by introducing
the relevant functor categories. Further details can be found in \cite{Demazure-Gabriel-English}
or \cite{Jantzen}, I~1-2.

From now on, $R$ will denote an arbitrary ring, except when specified
otherwise. Throughout this paper, $A$ will denote an Artinian local
ring with perfect residue field $k$. Let $R\text{-Alg}$ be the category
of $R$-algebras, and let $\text{Fun}/R$ denote the category of (covariant)
functors
\[
R\text{-Alg}\longrightarrow\text{Set}.
\]
Objects in $\text{Fun}/R$ are called $R$-\emph{functors} or \emph{functors
over $R$.} The category of affine schemes over $R$ (i.e., over $\Spec R$)
is then identified with the full subcategory of $\text{Fun}/R$ consisting
of \emph{representable} $R$-functors
\[
h_{S}:T\longmapsto\Hom_{R\text{-Alg}}(S,T),
\]
where $S$ is an $R$-algebra. Let $\text{Sch}/R$ denote the category
of schemes over $R$. Then $\text{Sch}/R$ embeds as a full subcategory
of $\text{Fun}/R$ via the functor $\mathbf{X}\mapsto h_{\mathbf{X}}$,
where $h_{\mathbf{X}}$ is given by 
\[
h_{\mathbf{X}}(S)=\Hom_{\text{Sch}/R}(\Spec S,\mathbf{X}),
\]
for any $R$-algebra $S$ (cf.~\cite{geoschemes}, Proposition VI-2).
In a similar way, any locally ringed space $\mathbf{X}$ over $\Spec R$
gives rise to an $R$-functor $h_{\mathbf{X}}$. When $\mathbf{X}$
is an affine scheme over $R$, we will write $R[\mathbf{X}]$ for
the $R$-algebra which represents $\mathbf{X}$.

We will now define the Greenberg functor. The main reference for this
are the original papers \cite{greenberg1,greenberg2}. A summary (in
the context of local principal ideal rings) can be found in \cite{Neron_models},
p.~276. Being a local ring, the characteristic of $A$ is either
equal to $p^{m+1}$, for some prime $p$ and some natural number $m$,
or it is equal to $0$, in which case we set $m=0$. For any integer
$n\geq0$, let $W_{n}:\mathbb{Z}\text{-Alg}\rightarrow\text{CRing}$
be the functor of ``$p$-typical'' truncated Witt vectors of length
$n$ (here $p$ is the prime given by the characteristic of $A$).
It is well-known that the functor $W_{n}$ is representable in $\text{Fun}/\mathbb{Z}$,
and we thus view it as an affine ring scheme over $\mathbb{Z}$. By
\cite{greenberg1}, 1, the ring $A$ has a canonical structure of
$W_{m}(k)$-algebra. In particular, in the equal characteristic case,
$\chara A$ is either $p$ or $0$, so $m=0$ and the ring $A$ is
a $k$-algebra. Furthermore, we can associate to $A$ an affine local
ring scheme $\mathbf{A}$ over $k$, such that for any perfect $k$-algebra
$R$, we have
\[
\mathbf{A}(R)=A\otimes_{W_{m}(k)}W_{m}(R),
\]
see \cite{Nicaise-Sebag}, 2.2.
\begin{defn}
The \emph{Greenberg functor }associated to $A$ is the functor
\begin{eqnarray*}
 & \mathcal{F}:\text{Fun}/A\longrightarrow\text{Fun}/k\\
 & \mathbf{X}\longmapsto\mathcal{F}\mathbf{X},
\end{eqnarray*}
where $\mathcal{F}\mathbf{X}$ is defined by
\[
(\mathcal{F}\mathbf{X})(R)=\mathbf{X}(\mathbf{A}(R)),
\]
for each $k$-algebra $R$.
\end{defn}
We will usually write $\mathcal{F}\mathbf{X}(R)$ instead of $(\mathcal{F}\mathbf{X})(R)$,
since there should be no confusion. We now show that the functor $\mathcal{F}$
associated to $A$ indeed coincides with the functor $F_{\mathbf{A}}$,
which was defined by Greenberg \cite{greenberg1} for schemes of finite
type over $A$. To this end, we recall Greenberg's functor $G_{\mathbf{A}}$
(cf.~\cite{greenberg1}, p.~634). Since Greenberg's original construction
is formulated in terms of schemes as locally ringed spaces, we will
for the moment turn to this point of view. Once the comparison between
our functor $\mathcal{F}$ and Greenberg's $F_{\mathbf{A}}$ is established,
we carry on using the functor of points approach. Let $\mathbf{Y}$
be a scheme over $k$, viewed as a locally ringed space with base
space $|\mathbf{Y}|$. Then $G_{\mathbf{A}}\mathbf{Y}$ is defined
to be the locally ringed space $(|\mathbf{Y}|,\mathcal{O})$, where
for any open subset $U\subseteq|\mathbf{Y}|$, the sheaf $\mathcal{O}$
is given by 
\[
\mathcal{O}(U)=\Hom_{k}(U,\mathbf{A})
\]
(morphisms as locally ringed spaces over $k$). Greenberg calls $\mathcal{O}$
``the sheaf of germs of $k$-morphisms from $\mathbf{Y}$ to $\mathbf{A}$'';
it is a sheaf of rings because $\mathbf{A}$ is a ring object. It
is shown in \cite{greenberg2} that $G_{\mathbf{A}}\mathbf{Y}$ is
a scheme over $A$, and moreover that $G_{\mathbf{A}}$ preserves
affine schemes. Thus, if $R$ is a $k$-algebra, then $G_{\mathbf{A}}(\Spec R)$
is an affine scheme over $A$. For $\mathbf{Y}=\Spec R$, the global
sections of $\mathcal{O}$ are just $\mathcal{O}(\Spec R)=\mathbf{A}(R)$,
so we have 
\[
G_{\mathbf{A}}(\Spec R)=\Spec\mathbf{A}(R).
\]
Suppose that $\mathbf{X}$ is a scheme over $A$ in the sense of locally
ringed spaces. From the definition of $\mathcal{F}$ above and the
Yoneda lemma, we then immediately obtain canonical isomorphisms
\begin{multline*}
(\mathcal{F}h_{\mathbf{X}})(R)=h_{\mathbf{X}}(\mathbf{A}(R))\\
\cong\Hom_{\text{Fun}/A}(h_{\mathbf{A}(R)},h_{\mathbf{X}})\cong\Hom_{\text{Fun}/A}(h_{\Spec\mathbf{A}(R)},h_{\mathbf{X}})\\
\cong\Hom_{\text{Sch}/A}(\Spec\mathbf{A}(R),\mathbf{X})\cong\mathbf{X}(\Spec\mathbf{A}(R))\\
=\mathbf{X}(G_{\mathbf{A}}(\Spec R)).
\end{multline*}
The isomorphisms are functorial in $R$, and it follows in particular
that whenever $\mathcal{F}h_{\mathbf{X}}$ is representable by a scheme
over $k$, it coincides with Greenberg's ``realization'' $F_{\mathbf{A}}\mathbf{X}$.
The key result is now the following
\begin{prop}
[Greenberg] Let $\mathbf{X}$ be a scheme locally of finite type
over $A$. Then $\mathcal{F}\mathbf{X}$ is representable as a scheme
locally of finite type over $k$.\end{prop}
\begin{proof}
This is essentially proved in \cite{greenberg1}, 4. Note that the
proof of Proposition~7 holds for any scheme locally of finite type
over $A$, and together with Corollary 1, indeed implies that $\mathcal{F}\mathbf{X}$
is a scheme locally of finite type over $k$.
\end{proof}
A nice consequence of the definition of $\mathcal{F}$ in terms of
functors of points is that it trivially preserves fibre products.
More precisely, let $\mathbf{X}$, $\mathbf{Y}$, and $\mathbf{Z}$
be objects in $\text{Fun}/A$. Then for any $k$-algebra $R$, we
have
\begin{multline*}
\mathcal{F}(\mathbf{X}\times_{\mathbf{Z}}\mathbf{Y})(R)=(\mathbf{X}\times_{\mathbf{Z}}\mathbf{Y})(\mathbf{A}(R))\\
=\mathbf{X}(\mathbf{A}(R))\times_{\mathbf{Z}(\mathbf{A}(R))}\mathbf{Y}(\mathbf{A}(R))=\mathcal{F}\mathbf{X}(R)\times_{\mathcal{F}\mathbf{Z}(R)}\mathcal{F}\mathbf{Y}(R)\\
=(\mathcal{F}\mathbf{X}\times_{\mathcal{F}\mathbf{Z}}\mathcal{F}\mathbf{Y})(R),
\end{multline*}
and hence 
\[
\mathcal{F}(\mathbf{X}\times_{\mathbf{Z}}\mathbf{Y})=\mathcal{F}\mathbf{X}\times_{\mathcal{F}\mathbf{Z}}\mathcal{F}\mathbf{Y}.
\]
It follows from this (see Section~\ref{sec:Transp-norm}) that $\mathcal{F}$
sends group objects in $\text{Fun}/A$ to group objects in $\text{Fun}/k$.
It also trivially preserves subfunctors. 

Following \cite{SGA3}, XI~1.1, we call an $R$-functor $\mathbf{X}$
\emph{formally smooth }(resp.~\emph{formally unramified}, resp.~\emph{formally
\'etale}) if for every $R$-algebra $S$, and every nilpotent ideal
$J$ in $S$, the induced map
\[
\mathbf{X}(S)\longrightarrow\mathbf{X}(S/J)
\]
is surjective (resp.~injective, resp.~bijective). Moreover $\mathbf{X}$
is called \emph{smooth over $R$ }(resp.\emph{~unramified over $R$},
resp.~\emph{\'etale over $R$}) if it satisfies the above condition,
and in addition is locally of finite presentation over $R$, that
is, if it commutes with filtered colimits. When $\mathbf{X}$ is representable
by a scheme, it commutes with filtered colimits if and only if it
is locally of finite presentation in the usual sense (cf.~\cite{EGA-IV},
III~8.14.2~c), note the contravariant statement that filtered colimits
are turned into filtered limits when working with affine schemes rather
than rings). In $ $\cite{EGA-IV}, IV~17.3.1, a map of schemes $\mathbf{X}\rightarrow\mathbf{Y}$
is defined to be \emph{smooth} if it is locally of finite presentation
and formally smooth. Another definition of smoothness is given in
\cite{EGA-IV}, II~6.8.1, but the two definitions are shown to be
equivalent in \cite{EGA-IV}, IV~17.5.2. This is sometimes referred
to as Grothendieck's infinitesimal criterion for smoothness, or the
infinitesimal lifting property. Note that if $R$ is Noetherian, ``(locally
of) finite presentation'' is equivalent to ``(locally of) finite
type''.

Greenberg has shown that if $\mathbf{X}$ is a smooth scheme over
$A$, then $\mathcal{F}\mathbf{X}$ is smooth over $k$ (see \cite{greenberg2},
p.~263, Corollary~1). The following is a generalisation of this
result to the functor of points setting. 
\begin{lem}
Let $\mathbf{X}$ be an $A$-functor. If $\mathbf{X}$ is smooth (resp.~unramified,
resp.~\'etale) over $A$, then $\mathcal{F}\mathbf{X}$ is smooth
(resp.~unramified, resp.~\'etale) over $k$.\end{lem}
\begin{proof}
The scheme $\mathbf{A}$ is isomorphic to an affine space over $k$
(see \cite{greenberg1}, 4). It is thus smooth over $k$. If $\mathbf{X}$
is locally of finite presentation (resp.~formally smooth) over $A$,
it therefore immediately follows that $\mathcal{F}\mathbf{X}=\mathbf{X}(\mathbf{A}(-))$
is locally of finite presentation (resp.~formally smooth) over $k$.
Moreover, let $S$ be a $k$-algebra, and $J$ a nilpotent ideal in
$S$. As we have just seen, the map $S\rightarrow S/J$ induces a
surjective map of rings $\mathbf{A}(S)\rightarrow\mathbf{A}(S/J)$.
Let $J_{\mathbf{A}(S)}$ be the kernel of the latter. Then, since
$\mathbf{A}$ is represented by the $k$-algebra $k[\mathbf{A}]$,
we have $J_{\mathbf{A}(S)}=\Hom_{R}(k[\mathbf{A}],J)$ (homomorphisms
of not-necessarily unital $R$-algebras), and so $J_{\mathbf{A}(S)}$
is nilpotent. If $\mathbf{X}$ is unramified (resp.~\'etale) over
$A$, then the morphism
\begin{align*}
(\mathcal{F}\mathbf{X})(S) & =\mathbf{X}(\mathbf{A}(S))\longrightarrow\mathbf{X}(\mathbf{A}(S)/J_{\mathbf{A}(S)})\\
 & \cong\mathbf{X}(\mathbf{A}(S/J))=(\mathcal{F}\mathbf{X})(S/J),
\end{align*}
is injective (resp.~bijective), so $\mathcal{F}\mathbf{X}$ is unramified
over $k$ (resp.~\'etale over $k$).
\end{proof}
Suppose that $\mathbf{G}$ is a group scheme over $R$, and let $S$
be an $R$-algebra. For any point $s\in\Spec S$, let $k(s)$ denote
the residue field at $s$, that is, the fraction field of $S/s$.
Then $k(s)$ is naturally an $R$-algebra, and we write $\mathbf{G}_{s}:=\mathbf{G}_{k(s)}$.
The \emph{connected component $\mathbf{G}^{\circ}$ of $\mathbf{G}$}
(cf.~\cite{SGA3}, VI$_{\textrm{A}}$~2, VI$_{\textrm{B}}$~3.1)
is the subgroup scheme of $\mathbf{G}$ whose $S$-points are given
by
\[
\mathbf{G}^{\circ}(S)=\{g\in\mathbf{G}(S)\mid g_{s}\in\mathbf{G}_{s}^{\circ}(k(s)),\text{ for all }s\in\Spec S\},
\]
where $g_{s}$ is the image of $g$ in $\mathbf{G}_{s}(k(s))\cong\mathbf{G}(k(s))$.
When $\mathbf{G}$ is smooth over $R$, the same is true for the connected
component $\mathbf{G}^{\circ}$ (cf.~\cite{SGA3}, VI$_{\textrm{B}}$~3.4,
3.10). We will be particularly interested in the case where $R$ is
an Artinian local base $A$ with residue field $k$. In this case
$\Spec A$ has a unique point $\mathfrak{m}$, and the $A$-points
of the connected component is simply given by
\[
\mathbf{G}^{\circ}(A)=\{g\in\mathbf{G}(A)\mid g_{\mathfrak{m}}\in\mathbf{G}_{k}^{\circ}(k)\}.
\]
Later on we will show that the Greenberg functor preserves connected
components of smooth group schemes.

\section{\label{sec:Transp-norm}Group scheme actions and transporters}

For any $R$-functor $\mathbf{X}$ and $R$-algebra $R'$, we write
$\mathbf{X}_{R'}$ for the base extension $\mathbf{X}\times_{R}R'$.
Given a map $S\rightarrow S'$ between two $R$-algebras and an element
$x\in\mathbf{X}(S)$, we denote by $x_{S'}$ the image of $x$ under
the induced map $\mathbf{X}(S)\rightarrow\mathbf{X}(S')$.

Let $\mathbf{G}$ be a group functor over $R$, that is, a group object
in the category $\text{Fun}/R$. This means that for any $R$-algebra
$S$, the set $\mathbf{G}(S)$ carries a group structure, or equivalently,
that there exist morphisms
\[
m:\mathbf{G}\times\mathbf{G}\longrightarrow\mathbf{G},\quad i:\mathbf{G}\longrightarrow\mathbf{G},\quad e:\mathbf{1}\longrightarrow\mathbf{G},
\]
satisfying the usual properties (here $\mathbf{1}$ denotes the terminal
object in $\text{Fun}/R$ which sends any $R$-algebra to the one-point
set). As we have noted earlier, the Greenberg functor $\mathcal{F}$
preserves fibre products. It also obviously sends the terminal object
in $\text{Fun}/A$ to the terminal object in $\text{Fun}/k$. Thus,
if $\mathbf{G}$ is a group functor over $A$ with maps $m,i,e$,
then $\mathcal{F}\mathbf{G}$ is a group functor over $k$ with maps
$\mathcal{F}(m),\mathcal{F}(i),\mathcal{F}(e)$. 

An \emph{action} of $\mathbf{G}$ on an $R$-functor $\mathbf{X}$
is a morphism
\[
\mathbf{G}\times\mathbf{X}\longrightarrow\mathbf{X},
\]
such that for each $R$-algebra $S$, the induced morphism $\mathbf{G}(S)\times\mathbf{X}(S)\rightarrow\mathbf{X}(S)$
defines an action of the group $\mathbf{G}(S)$ on the set $\mathbf{X}(S)$,
in the usual sense. If $\mathbf{G}$ is an $A$-group functor acting
on an $A$-functor $\mathbf{X}$, then the induced morphism
\[
\mathcal{F}\mathbf{G}\times\mathcal{F}\mathbf{X}\longrightarrow\mathcal{F}\mathbf{X},
\]
defines an action of the $k$-group functor $\mathcal{F}\mathbf{G}$
on $\mathcal{F}\mathbf{X}$. One of the most important examples of
an action is that of $\mathbf{G}$ acting on itself by conjugation,
that is, the morphism $\gamma_{\mathbf{G}}:\mathbf{G}\times\mathbf{G}\rightarrow\mathbf{G}$
such that for each $R$-algebra $S$, the map $\gamma_{\mathbf{G}}(S):\mathbf{G}(S)\times\mathbf{G}(S)\rightarrow\mathbf{G}(S)$
is given by $(g,g')\mapsto gg'g^{-1}$. Suppose that $\mathbf{G}$
is a group functor with maps $m,i,e$, as above, and let $p_{1}:\mathbf{G}\times\mathbf{G}\rightarrow\mathbf{G}$
be the first projection map. Then the conjugation action of $\mathbf{G}$
on itself is given by the composition of the maps 
\[
\mathbf{G}\times\mathbf{G}\xrightarrow{m\times i\circ p_{1}}\mathbf{G}\times\mathbf{G}\xrightarrow{\, m\,}\mathbf{G}.
\]
It then immediately follows that the Greenberg functor preserves the
conjugation action. More precisely, if $\mathbf{G}$ is a group functor
over $A$, then 
\[
\mathcal{F}(\gamma_{\mathbf{G}})=\gamma_{\mathcal{F}(\mathbf{G})}.
\]
For any $R$-functors $\mathbf{X}$ we consider the \emph{functor
of automorphisms}, written $\underline{\Aut}_{R}(\mathbf{\mathbf{X}})$
or simply $\underline{\Aut}(\mathbf{\mathbf{X}})$, when there is
no confusion about the base ring. This is the $R$-functor defined
by
\[
\underline{\Aut}(\mathbf{\mathbf{X}})(S)=\Aut_{\text{Fun}/S}(\mathbf{X}_{S}),
\]
for any $R$-algebra $S$. An action of $\mathbf{G}$ on $\mathbf{X}$
then gives rise to a morphism $\mathbf{G}\rightarrow\underline{\Aut}(\mathbf{X})$.
For example, the conjugation action of $\mathbf{G}$ on itself gives
rise to the morphism $\mathbf{G}\rightarrow\underline{\Aut}(\mathbf{G})$
such that for any $R$-algebra $S$, the map $\mathbf{G}(S)\rightarrow\Aut(\mathbf{G}_{S})$
is given by $\mathbf{}$$g\mapsto\ad(g)$, where $\ad(g):\mathbf{G}_{S}\rightarrow\mathbf{G}_{S}$
is the morphism defined by
\[
\ad(g)(S'):x\longmapsto g_{S'}xg_{S'}^{-1},\ \text{for }x\in\mathbf{G}(S')\text{ and any map }S\rightarrow S'.
\]
If $\mathbf{G}$ is a group functor over $A$, the Greenberg functor
preserves the conjugation action, and thus
\[
\mathcal{F}(\ad(g))=\ad(g):\mathcal{F}\mathbf{G}\longrightarrow\mathcal{F}\mathbf{G},
\]
for any $g\in\mathbf{G}(A)=\mathcal{F}\mathbf{G}(k)$.

If $f:\mathbf{X}\rightarrow\mathbf{Y}$ is a morphism of $R$-functors,
we write $f(\mathbf{X})$ for the image functor, given by
\[
f(\mathbf{X})(S):=\im(f(S)(\mathbf{X}(S))),
\]
for any $R$-algebra $S$. Suppose that $\mathbf{X}$ is an $R$-functor,
and $\mathbf{Y}$ and $\mathbf{Z}$ are two subfunctors of $\mathbf{X}$
given by inclusions $i:\mathbf{Y}\hookrightarrow\mathbf{X}$ and $j:\mathbf{Z}\hookrightarrow\mathbf{X}$,
respectively. We write $\mathbf{Y}=\mathbf{Z}$ and say that $\mathbf{Y}$
and $\mathbf{Z}$ are \emph{equal as subfunctors of $\mathbf{X}$},
if there exists an isomorphism $c:\mathbf{Y}\rightarrow\mathbf{Z}$
such that $i=j\circ c$.
\begin{defn}
Let $\mathbf{G}$ be an $R$-group functor acting on an $R$-functor
$\mathbf{X}$, and let $\alpha:\mathbf{G}\rightarrow\underline{\Aut}(\mathbf{X})$
be the corresponding morphism. Let $\mathbf{Y}$ and $\mathbf{Z}$
be two subfunctors of $\mathbf{X}$. Let $S$ be an arbitrary $R$-algebra.
Define the \emph{strict transporter $\T_{\mathbf{G}}(\mathbf{Y},\mathbf{Z})$
from $\mathbf{Y}$ to $\mathbf{Z}$ in $\mathbf{X}$}, to be the subfunctor
of $\mathbf{G}$ whose $S$-points are given by
\begin{align*}
\T_{\mathbf{G}}(\mathbf{Y},\mathbf{Z})(S) & =\{g\in\mathbf{G}(S)\mid\alpha(S)(g)(\mathbf{Y}_{S})=\mathbf{Z}_{S}\}\\
 & =\{g\in\mathbf{G}(S)\mid\alpha(S')(g_{S'})(S')(\mathbf{Y}(S'))=\mathbf{Z}(S'),\text{ for any }S\rightarrow S'\}.
\end{align*}
In particular, if $\mathbf{G}$ acts on itself by conjugation, $\mathbf{X}=\mathbf{G}$,
and $\mathbf{Y}=\mathbf{Z}$, we write $\N_{\mathbf{G}}(\mathbf{Y})$
for the strict transporter from $\mathbf{Y}$ to $\mathbf{Y}$ in
$\mathbf{G}$, and call it the \emph{normaliser of $\mathbf{Y}$ in
$\mathbf{G}$}. Its $S$-points are thus given by
\begin{align*}
\N_{\mathbf{G}}(\mathbf{Y})(S) & =\{g\in\mathbf{G}(S)\mid\ad(g)(\mathbf{Y}_{S})=\mathbf{Y}_{S}\}\\
 & =\{g\in\mathbf{G}(S)\mid g_{S'}\mathbf{Y}(S')g_{S'}^{-1}=\mathbf{Y}(S'),\text{ for any }S\rightarrow S'\}.
\end{align*}
\end{defn}
\begin{rem*}
Transporters are defined in \cite{SGA3}, VI$_{\text{B}}$~6.1 in
the case where $\mathbf{G}$ acts on itself by conjugation, and where
$\mathbf{Y}$ and $\mathbf{Z}$ are subfunctors of $\mathbf{G}$.
One may also consider the (not necessarily strict) transporter, whose
$S$-points are defined by an inclusion rather than an equality. The
two types of transporters coincide in the case where $\mathbf{G}$
is a scheme acting on itself by conjugation, $\mathbf{Y}=\mathbf{Z}$
is a subscheme of $\mathbf{G}$, and either $\mathbf{Y}$ is of finite
presentation over $R$, or $\T_{\mathbf{G}}(\mathbf{Y},\mathbf{Y})=\N_{\mathbf{G}}(\mathbf{H})$
is representable$ $ by a scheme of finite presentation over $R$
(cf.~\cite{SGA3}, VI$_{\text{B}}$~6.4). We will only be interested
in normalisers in situations where both of these conditions are satisfied,
so we will not distinguish between the normaliser and the strict normaliser.

One may define centralisers in a similar way, but these will play
no role in this paper.
\end{rem*}
From now on, suppose that $k$ is an algebraically closed field. Let
$\mathbf{X}$ be an affine variety over $k$, that is, a (not necessarily
irreducible) scheme which is affine, reduced, and of finite type over
$k$). Recall that Hilbert's Nullstellensatz implies that a closed
reduced subvariety $\mathbf{Y}$ of $\mathbf{X}$ is determined by
its set of points $\mathbf{Y}(k)\subseteq\mathbf{X}(k)$. More precisely,
if $\mathbf{Y}$ and $\mathbf{Z}$ are two closed reduced subvarieties
of $\mathbf{X}$, defined by the radical ideals $I$ and $J$ of $k[\mathbf{X}]$,
respectively, then $\mathbf{Y}(k)=\mathbf{Z}(k)$ as subsets of $\mathbf{X}(k)$,
implies that $I=J$. The purpose of the following result is to prove
a generalisation of this. 
\begin{prop}
\label{pro:Nullstellensatz}Let $A$ be an Artinian local ring with
algebraically closed residue field $k$. Let $\mathbf{X}$ be an affine
scheme of finite type over $A$, and let $\mathbf{Y}$ and $\mathbf{Y}'$
be closed smooth subschemes of $\mathbf{X}$. Then $\mathbf{Y}=\mathbf{Y}'$
if and only if $\mathbf{Y}(A)=\mathbf{Y}'(A)$. \end{prop}
\begin{proof}
The ``only if'' part is trivial. Assume hence that $\mathbf{Y}(A)=\mathbf{Y}'(A)$.
Since $\mathbf{Y}$ and $\mathbf{Y}'$ are smooth over $A$, the canonical
maps $\mathbf{Y}(A)\rightarrow\mathbf{Y}(k)$ and $\mathbf{Y}'(A)\rightarrow\mathbf{Y}'(k)$
are surjective. Since they are both also restrictions of the map $\mathbf{X}(A)\rightarrow\mathbf{X}(k)$,
we obtain $\mathbf{Y}(k)=\mathbf{Y}'(k)$. Let $A[\mathbf{X}]$ be
the affine algebra of $\mathbf{X}$, which we identify with an algebra
of polynomial functions on $\mathbf{X}(A)$ by embedding $\mathbf{X}$
in affine space. For every subset $V\subseteq\mathbf{X}(A)$, let
\[
\mathcal{I}(V)=\{f\in A[\mathbf{X}]\mid f(x)=0,\text{ for all }x\in V\}.
\]
On the other hand, for any ideal $J$ in $A[\mathbf{X}]$, let
\[
\mathcal{V}(J)=\{p\in\mathbf{X}(A)\mid f(p)=0\text{ for all }f\in J\}=\Hom_{A}(A[\mathbf{X}]/J,A).
\]
 If $\overline{V}\subseteq\mathbf{X}_{k}(k)$ and $\overline{J}$
is an ideal in $A[\mathbf{X}]\otimes k$, then we write $\mathcal{I}_{k}(\overline{V})$
and $\mathcal{V}_{k}(\overline{J})$ for the analogous objects in
$A[\mathbf{X}]\otimes k$ and $\mathbf{X}_{k}(k)$, respectively.
$ $ Let $I$ and $I'$ be the ideals in $A[\mathbf{X}]$ defining
$\mathbf{Y}$ and $\mathbf{Y}'$, respectively. Note that we have
$\mathcal{V}(I)=\mathbf{Y}(A)$ and $\mathcal{V}(I')=\mathbf{Y}'(A)$.
Let $\mathfrak{m}$ be the maximal ideal in $A$. If $J$ is an ideal
in $A[\mathbf{X]}$, write $\overline{J}$ for its image in $A[\mathbf{X}]\otimes k\cong A[\mathbf{X}]/\mathfrak{m}A[\mathbf{X}]$.
Since the fibres $\mathbf{Y}_{k}$ and $\mathbf{Y}'_{k}$ are reduced,
$\overline{I}$ and $\overline{I}\vphantom{I}'$ $ $are radical ideals
of $A[\mathbf{X}]\otimes k$. We have $\mathcal{V}_{k}(\overline{I})=\mathbf{Y}_{k}(k)$
and $\mathcal{V}_{k}(\overline{I}\vphantom{I}')=\mathbf{Y}'_{k}(k)$,
and thus the Nullstellensatz yields 
\[
\overline{I}=\mathcal{I}_{k}(\mathbf{Y}_{k}(k))=\mathcal{I}_{k}(\mathbf{Y}'_{k}(k))=\overline{I}\vphantom{I}'.
\]
Denote by $\overline{\mathcal{V}(I)}$ the image of $\mathcal{V}(I)$
under the map $\mathbf{Y}(A)\rightarrow\mathbf{Y}(k)$. Since the
latter is surjective, we have $\overline{\mathcal{V}(I)}=\mathbf{Y}_{k}(k)$.
Hence, 
\begin{align*}
\overline{\mathcal{I}(\mathcal{V}(I))} & \subseteq\{\overline{f}\in A[\mathbf{X}]\otimes k\mid\overline{f}(\overline{x})=0,\text{ for all }\overline{x}\in\overline{\mathcal{V}(I)}\}\\
 & =\{\overline{f}\in A[\mathbf{X}]\otimes k\mid\overline{f}(\overline{x})=0,\text{ for all }\overline{x}\in\mathbf{Y}_{k}(k)\}\\
 & =\mathcal{I}_{k}(\mathbf{Y}_{k}(k))=\overline{I}.
\end{align*}
On the other hand, we obviously have $I\subseteq\mathcal{I}(\mathcal{V}(I))$,
hence $\overline{I}\subseteq\overline{\mathcal{I}(\mathcal{V}(I))}$,
and so $\overline{I}=\overline{\mathcal{I}(\mathcal{V}(I))}$. In
the same way we obtain $\overline{I}\vphantom{I}'=\overline{\mathcal{I}(\mathcal{V}(I'))}$.
This implies that
\[
\mathcal{I}(\mathcal{V}(I))=I+\mathfrak{m}\mathcal{I}(\mathcal{V}(I)),\quad\text{and}\quad\mathcal{I}(\mathcal{V}(I'))=I'+\mathfrak{m}\mathcal{I}(\mathcal{V}(I')).
\]
Since $A$ is Noetherian, any ideal in $A[\mathbf{X}]$ is finitely
generated, and we can apply Nakayama's lemma to get
\[
I=\mathcal{I}(\mathcal{V}(I)),\quad\text{and}\quad I'=\mathcal{I}(\mathcal{V}(I')).
\]
From the hypothesis $\mathbf{Y}(A)=\mathbf{Y}'(A)$ we then conclude
that
\[
I=\mathcal{I}(\mathcal{V}(I))=\mathcal{I}(\mathbf{Y}(A))=\mathcal{I}(\mathbf{Y}'(A))=\mathcal{I}(\mathcal{V}(I'))=I',
\]
and so $\mathbf{Y}=\mathbf{Y}'$.\end{proof}
\begin{rem*}
The author has been informed that Proposition~\ref{pro:Nullstellensatz}
is a consequence of a schematic density statement, due to Grothendieck
(\cite{EGA-IV}, III~11.10.9). We have however given a direct and
self-contained proof in the case that is of interest to us here. 
\end{rem*}
From now on, assume that $A$ is an Artinian local ring with algebraically
closed residue field $k$. We have the Greenberg functor $\mathcal{F}$
associated to $A$. We recall a well-known fact which will be used
several times in what follows: Suppose that $\mathbf{G}$ is an affine
group scheme of finite type over an algebraically closed field $k$.
Then $\mathbf{G}$ is smooth over $k$ if and only if it is reduced
over $k$ (cf.~\cite{waterhouse}, 11.6).
\begin{prop}
\label{pro:F and Tran}Let $\mathbf{G}$ be an affine group scheme
of finite type over $A$, acting on an affine scheme $\mathbf{X}$
of finite type over $A$, and let $\mathbf{Y}$ and $\mathbf{Z}$
be closed smooth subschemes of $\mathbf{X}$. Let $\mathcal{F}\mathbf{G}$
act on $\mathcal{F}\mathbf{X}$ via the action induced from that of
$\mathbf{G}$ on $\mathbf{X}$. Then 
\[
\mathcal{F}\T_{\mathbf{G}}(\mathbf{Y},\mathbf{Z})(k)=\T_{\mathcal{F}\mathbf{G}}(\mathcal{F}\mathbf{Y},\mathcal{F}\mathbf{Z})(k).
\]
\end{prop}
\begin{proof}
Let $\alpha:\mathbf{G}\rightarrow\underline{\Aut}(\mathbf{X})$ be
the action of $\mathbf{G}$ on $\mathbf{X}$. Then 
\[
\mathcal{F}\T_{\mathbf{G}}(\mathbf{Y},\mathbf{Z})(k)=\T_{\mathbf{G}}(\mathbf{Y},\mathbf{Z})(A)=\{g\in\mathbf{G}(A)\mid\alpha(A)(g)(\mathbf{Y})=\mathbf{Z}\}.
\]
Since $\alpha(A)(g)(\mathbf{Y})$ and $\mathbf{Z}$ are both closed
smooth subschemes of $\mathbf{G}$, Proposition~\ref{pro:Nullstellensatz}
implies that the condition $\alpha(A)(g)(\mathbf{Y})=\mathbf{Z}$
is equivalent to $(\alpha(A)(g))(A)(\mathbf{Y}(A))=\mathbf{Z}(A)$,
and so
\[
\mathcal{F}\T_{\mathbf{G}}(\mathbf{Y},\mathbf{Z})(k)=\{g\in\mathbf{G}(A)\mid(\alpha(A)(g))(A)(\mathbf{Y}(A))=\mathbf{Z}(A)\}.
\]
In the same way, the Nullstellensatz implies that
\begin{align*}
\T_{\mathcal{F}\mathbf{G}}(\mathcal{F}\mathbf{Y},\mathcal{F}\mathbf{Z})(k) & =\{g\in\mathcal{F}\mathbf{G}(k)\mid(\mathcal{F}(\alpha)(k)(g))(k)(\mathcal{F}\mathbf{Y}(k))=\mathcal{F}\mathbf{Z}(k)\}\\
 & =\{g\in\mathbf{G}(A)\mid(\alpha(A)(g))(A)(\mathbf{Y}(A))=\mathbf{Z}(A)\},
\end{align*}
and the result is proved.
\end{proof}
In the above proof we have used the fact that (under the hypotheses
of Proposition~\ref{pro:F and Tran}) the Nullstellensatz implies
that
\[
\T_{\mathcal{F}\mathbf{G}}(\mathcal{F}\mathbf{Y},\mathcal{F}\mathbf{Z})(k)=\T_{\mathcal{F}\mathbf{G}(k)}(\mathcal{F}\mathbf{Y}(k),\mathcal{F}\mathbf{Z}(k)),
\]
where the right-hand side is the set-theoretical strict transporter,
defined in the obvious way. Results of this type for normalisers and
centralisers over algebraically closed fields are well-known and appear
in, for example, \cite{Demazure-Gabriel-English} II, \S5, 4.1 and
\cite{Jantzen}, I, 2.6. More generally, Proposition~\ref{pro:Nullstellensatz}
implies that if $\mathbf{G}$ is affine of finite type over $A$ and
the subschemes $\mathbf{Y}$ and $\mathbf{Z}$ are smooth, then 
\[
\T_{\mathbf{G}}(\mathbf{Y},\mathbf{Z})(A)=\T_{\mathbf{G}(A)}(\mathbf{Y}(A),\mathbf{Z}(A)),
\]
where the right-hand side is the set-theoretical strict transporter.
\begin{cor}
\label{cor:F - Norm}Let $\mathbf{G}$ be an affine group scheme of
finite type over $A$, and let $\mathbf{H}$ be a closed smooth subgroup
scheme. Then
\[
\mathcal{F}\N_{\mathbf{G}}(\mathbf{H})(k)=\N_{\mathcal{F}\mathbf{G}}(\mathcal{F}\mathbf{H})(k).
\]
\end{cor}
\begin{proof}
We have seen that $\mathcal{F}$ transforms the conjugation action
of a group functor $\mathbf{G}$ over $A$ on itself, into the conjugation
action of $\mathcal{F}\mathbf{G}$ on itself. Now apply Theorem~\ref{pro:F and Tran}
with $\mathbf{Y}=\mathbf{Z}=\mathbf{H}$, and $\mathbf{G}$ acting
by conjugation.
\end{proof}
A group scheme $\mathbf{G}$ over $R$ is called \emph{reductive}
if it is affine and smooth over $R$, and if all its geometric fibres
$\mathbf{G}_{\overline{k(s)}}$ are connected reductive groups in
the classical sense (cf.~\cite{Demazure}, 2.1 or \cite{SGA3}, XIX~2.7).
If $\mathbf{G}$ is a reductive group over $R$, a \emph{maximal torus}
(resp.~a \emph{Borel subgroup}, resp.~a \emph{parabolic subgroup})
of $\mathbf{G}$ is a smooth subgroup scheme $\mathbf{H}$, such that
each geometric fibre $\mathbf{H}_{\overline{k(s)}}$ is a maximal
torus (resp.~a Borel subgroup, resp.~a parabolic subgroup) of $\mathbf{G}_{\overline{k(s)}}$,
in the classical sense (cf.~\cite{SGA3}, XV~6.1).

The following lemma gives the most important situations where the
normaliser is representable by a closed smooth subscheme. This provides
the cases for which we will subsequently apply Corollary~\ref{cor:F - Norm}.
\begin{lem}
\label{lem:Norm T - P}Let $\mathbf{G}$ be a reductive group scheme
over $A$. Let $\mathbf{T}$ be a maximal torus of $\mathbf{G}$,
and let $\mathbf{P}$ be a parabolic subgroup of $\mathbf{G}$. Then
$\N_{\mathbf{G}}(\mathbf{T})$ and $\N_{\mathbf{G}}(\mathbf{P})$
are representable by closed smooth subschemes of $\mathbf{G}$, respectively.
Moreover, we have
\[
\N_{\mathbf{G}}(\mathbf{T})^{\circ}=\mathbf{T},\quad\text{and}\quad\N_{\mathbf{G}}(\mathbf{P})=\mathbf{P}.
\]
\end{lem}
\begin{proof}
All the statements concerning the representability of $\N_{\mathbf{G}}(\mathbf{T})$
and $\N_{\mathbf{G}}(\mathbf{P})$ follow from \cite{SGA3}, XII~7.9
(see also XXII~5.3.10). The fact that $\N_{\mathbf{G}}(\mathbf{T})^{\circ}=\mathbf{T}$
is part of \cite{SGA3}, XII~7.9 (see also XXII~5.2.2). Finally,
the statement $\N_{\mathbf{G}}(\mathbf{P})=\mathbf{P}$ is contained
in \cite{SGA3}, XXII~5.8.5 (see also XIV~4.8-4.8.1 and XXVI~1.2).
\end{proof}

\section{\label{sec:Alg grps}The associated algebraic groups}

We keep our assumption that $A$ is an Artinian local ring with algebraically
closed residue field $k$. Let $\mathfrak{m}$ be the maximal ideal
of $A$. Let $\mathbf{G}$ be an affine smooth group scheme over $A$.
Define the group
\[
G=(\mathcal{F}\mathbf{G})(k),
\]
and for any integer $r\geq0$, let
\[
G_{r}=\mathcal{F}(\mathbf{G}\times_{A}A/\mathfrak{m}^{r})(k).
\]
Note that for $r=0$ the ring $A/\mathfrak{m}^{r}$ is the trivial
ring $\{0=1\}$, so $G_{0}$ consists of exactly one point. On the
other hand, if $\mathfrak{m}^{r}=0$, then $G=G_{r}$. Since $\mathbf{G}$
is smooth it follows from the infinitesimal criterion for smoothness
that for any integers $r\geq r'\geq0$, the canonical reduction map
$A/\mathfrak{m}^{r}\rightarrow A/\mathfrak{m}^{r'}$ induces a surjective
homomorphism $\rho_{r,r'}:G_{r}\rightarrow G_{r'}$. The kernel of
$\rho_{r,r'}$ is denoted by $G_{r}^{r'}$. In particular, when $G=G_{r}$
we write $\rho_{r'}$ for $\rho_{r,r'}$ and $G^{r'}$ for $G_{r}^{r'}$.

We will refer to affine smooth group schemes over $k$ as \emph{linear
algebraic groups}. This coincides with the classical notion of linear
algebraic group, as defined for example in \cite{humphreys}. Hence
$G_{r}$ is a linear algebraic group over $k$, for any $r\geq0$.
If $\mathbf{H}$ is a subgroup scheme of $\mathbf{G}$, we will write
$H$ for the corresponding closed subgroup $\mathcal{F}\mathbf{H}(k)$
of $G$.
\begin{lem}
\label{lem:Conn-unip}Suppose that we have an exact sequence of linear
algebraic groups
\[
1\longrightarrow K\longrightarrow L\xrightarrow{\ \alpha\ }M\longrightarrow1.
\]
If $K$ and $M$ are connected (resp.~unipotent), then $L$ is connected
(resp.~unipotent).\end{lem}
\begin{proof}
Assume that $K$ and $M$ are connected. Then $\alpha(L^{\circ})=\alpha(L)^{\circ}=M$,
and so we have an exact sequence $1\rightarrow K\rightarrow L^{\circ}\xrightarrow{\alpha}M\rightarrow1$.
Thus, for $x\in L$, there exists a $y\in L^{\circ}$ such that $\alpha(x)=\alpha(y)$.
Since $K$ lies in $L^{\circ}$, we must have $x\in L^{\circ}$. Now
assume that $K$ and $M$ are unipotent. Then the set $L_{u}$ of
unipotent elements satisfies $\alpha(L_{u})=\alpha(L)_{u}=M$, so
$\alpha$ maps $L_{u}$ surjectively onto $M$, and $K=\{z\in L_{u}\mid\alpha(z)=1\}$.
Thus, for $x\in L$, there exists a $y\in L_{u}$ such that $\alpha(x)=\alpha(y)$.
Let $x=x_{s}x_{u}$ and $y=y_{s}y_{u}$ be the Jordan decomposition
of $x$ and $y$, respectively. Then $\alpha(x_{s})=\alpha(x)_{s}=\alpha(y)_{s}=\alpha(y_{s})=1$,
so $x_{s}\in K$. Since $x_{s}$ is semisimple and $K$ consists of
unipotent elements, we have $x_{s}=1$, and so $x\in L_{u}$.
\end{proof}
Assume that $\alpha:L\rightarrow M$ is a surjective morphism of linear
algebraic groups with connected kernel $K$. We then get an exact
sequence 
\[
1\longrightarrow K\longrightarrow\alpha^{-1}(M^{\circ})\xrightarrow{\ \alpha\ }M^{\circ}\longrightarrow1,
\]
and it follows from the above lemma that $\alpha^{-1}(M^{\circ})$
is connected. Since $L^{\circ}\subseteq\alpha^{-1}(M^{\circ})\subseteq L$,
we must in fact have $L^{\circ}=\alpha^{-1}(M^{\circ})$. 

In \cite{Neron_models}, p.~277, it is stated (without proof) that
the Greenberg functor respects connected components of smooth group
schemes. The following result provides a proof of this. Note that
this could also be proved using \cite{greenberg2}, p.~264, Corollary~2,
mentioned in the introduction.
\begin{lem}
\label{lem:Conn comp}Let $\mathbf{G}$ be an affine smooth group
scheme over $A$. Then 
\[
\mathcal{F}(\mathbf{G}^{\circ})=(\mathcal{F}\mathbf{G})^{\circ}.
\]
\end{lem}
\begin{proof}
Since $\mathbf{G}$ is smooth over $A$, the same is true for $\mathbf{G}^{\circ}$.
As we have noted earlier, any smooth group scheme over an algebraically
closed field is reduced. By the Nullstellensatz it is then enough
to show that
\[
\mathcal{F}(\mathbf{G}^{\circ})(k)=(\mathcal{F}\mathbf{G})^{\circ}(k).
\]
Unraveling the definitions, we have
\[
\mathcal{F}(\mathbf{G}^{\circ})(k)=\mathbf{G}^{\circ}(A)=\{g\in\mathbf{G}(A)\mid g_{k}\in\mathbf{G}_{k}^{\circ}(k)\}=\{g\in\mathcal{F}\mathbf{G}(k)\mid\rho_{1}(g)\in\mathbf{G}_{k}^{\circ}(k)\}.
\]
Write $G^{\circ}$ for $(\mathcal{F}\mathbf{G}^{\circ})(k)$. Since
$\mathbf{G}^{\circ}$ is an affine smooth group scheme over $A$ with
connected fibre, it follows from Greenberg's structure theorem \cite{greenberg2},
2 (see also 3), that each kernel $(G^{\circ})_{r+1}^{r}$ is connected.
For every integer $r\geq0$, we have an exact sequence
\[
1\longrightarrow(G^{\circ})_{r+1}^{r}\longrightarrow(G^{\circ})_{r+1}^{1}\xrightarrow{\,\rho_{r+1,r}\,}(G^{\circ})_{r}^{1}\longrightarrow1.
\]
By repeated use of Lemma~\ref{lem:Conn-unip}, using the fact that
$(G^{\circ})_{1}\cong\mathbf{G}_{k}^{\circ}(k)$ is connected, it
follows that the kernel $(G^{\circ})^{1}$ is connected. Hence $G^{\circ}$
sits in the exact sequence
\[
1\longrightarrow(G^{\circ})^{1}\longrightarrow G^{\circ}\xrightarrow{\,\rho_{1}\,}(G^{\circ})_{1}\longrightarrow1,
\]
and it follows from Lemma~\ref{lem:Conn-unip} that $G^{\circ}$
is connected. Since $G^{\circ}=\rho_{1}^{-1}(\mathbf{G}_{k}^{\circ}(k))$,
it contains the connected component $(\mathcal{F}\mathbf{G})^{\circ}(k)$
of $(\mathcal{F}\mathbf{G})(k)=G$, and the maximality of the latter
forces $G^{\circ}=(\mathcal{F}\mathbf{G})^{\circ}(k)$. 
\end{proof}
For any linear algebraic group $G$, let $R_{u}(G)$ denote its unipotent
radical.
\begin{prop}
\label{pro:Unip rad and B}Let $\mathbf{G}$ be an affine smooth group
scheme over $A$, such that the fibre $\mathbf{G}_{k}$ is connected
. Then $G$ is connected, and $R_{u}(G)=\rho_{1}^{-1}(R_{u}(G_{1}))$.
In particular, if $\mathbf{G}$ is a reductive group scheme over $A$,
then $G$ is connected and $R_{u}(G)=G^{1}$. Moreover, let $\mathbf{B}$
be a Borel subgroup in $\mathbf{G}$. Then $BG^{1}$ is a Borel subgroup
in $G$.\end{prop}
\begin{proof}
The connectedness of $G$ follows from Lemma~\ref{lem:Conn comp}.
It is well-known that surjective morphisms between linear algebraic
groups respect unipotent radicals; hence, $\rho_{1}(R_{u}(G))=R_{u}(G_{1})$.
As in the proof of Lemma~\ref{lem:Conn comp}, it follows from Greenberg's
structure theorem and Lemma~\ref{lem:Conn-unip} that $G^{1}$ is
connected and unipotent. By definition, $R_{u}(G)$ is the biggest
closed connected unipotent normal subgroup of $G$, so $G^{1}$ sits
inside $R_{u}(G)$, and we have an exact sequence
\[
1\longrightarrow G^{1}\longrightarrow R_{u}(G)\xrightarrow{\ \rho_{1}\ }R_{u}(G_{1})\longrightarrow1.
\]
Since $R_{u}(G)$ and $R_{u}(G_{1})$ are both connected, the discussion
following Lemma~\ref{lem:Conn-unip} shows that $R_{u}(G)=\rho_{1}^{-1}(R_{u}(G_{1}))$.
If $\mathbf{G}$ is reductive over $A$, it has connected reductive
fibre by definition, so in this case, $R_{u}(G_{1})=\{1\}$.

Surjective maps between connected linear algebraic groups are known
to send Borel subgroups to Borel subgroups (see \cite{humphreys},
21.3C), so $BG^{1}=\rho_{1}^{-1}(B_{1})$ contains a Borel subgroup
of $G$. Since $BG^{1}$ is an extension of the Borel subgroup $B_{1}$
by the unipotent (hence solvable) group $G^{1}$, it is itself solvable.
Since it contains a Borel subgroup, it must in fact equal this Borel. 
\end{proof}
Recall that a \emph{Cartan subgroup} of a linear algebraic group $G$
is defined to be the centraliser of a maximal torus in $G$. Cartan
subgroups are closed, connected, nilpotent groups, and if $G$ is
reductive, its set of Cartan subgroups coincides with its set of maximal
tori. It is well-known that any two Cartan subgroups of $G$ are conjugate
in $G$. The following is a useful characterisation of Cartan subgroups.
\begin{lem}
\label{lem:Cartan-criterion}Let $C$ be a closed, connected, nilpotent
subgroup of a linear algebraic group $G$, and suppose that $\N_{G}(C)^{\circ}=C$.
Then $C$ is a Cartan subgroup.\end{lem}
\begin{proof}
See \cite{borel}, 12.6. 
\end{proof}
Suppose that $k$ is the field of definition of the groups in the
above lemma. We remark that $\N_{G}(C)$ should be thought of as the
group $\N_{G}(C)_{\mathrm{red}}(k)$, where $\N_{G}(C)$ is the scheme
theoretic normaliser, and $\N_{G}(C)_{\mathrm{red}}$ is the reduced
group scheme associated to $\N_{G}(C)$. At the scheme level this
distinction is important because if $H$ is a closed subgroup of $G$,
the scheme theoretic normaliser $\N_{G}(H)$ may not be reduced. However,
we always have $\N_{G}(H)_{\mathrm{red}}(k)=\N_{G}(H)(k)$, since
if $\N_{G}(H)$ is represented by the ring $R$, then $\N_{G}(H)_{\mathrm{red}}$
is represented by $R/\mathrm{nil}(R)$, where $\mathrm{nil}(R)$ is
the nilradical, and every homomorphism $R\rightarrow k$ factors through
$R/\mathrm{nil}(R)$. 

We can now give the proof of our main result.
\begin{thm}
\label{thm:Cartan}Let $\mathbf{G}$ be a reductive group scheme over
$A$, and let $\mathbf{T}$ be a maximal torus in $\mathbf{G}$. Then
$T$ is a Cartan subgroup of $G$, and the map $\mathbf{T}\mapsto T$
is a bijection between the set of maximal tori in $\mathbf{G}$ and
the set of Cartan subgroups of $G$. \end{thm}
\begin{proof}
By Corollary~\ref{cor:F - Norm}, Lemma~\ref{lem:Norm T - P}, and
Lemma~\ref{lem:Conn comp}, we have
\begin{multline*}
\N_{G}(T)^{\circ}=\N_{\mathcal{F}\mathbf{G}(k)}(\mathcal{F}\mathbf{T}(k))^{\circ}\\
=\N_{\mathcal{F}\mathbf{G}}(\mathcal{F}\mathbf{T})(k)^{\circ}=\N_{\mathcal{F}\mathbf{G}}(\mathcal{F}\mathbf{T})^{\circ}(k)=\mathcal{F}(\N_{\mathbf{G}}(\mathbf{T})^{\circ})(k)\\
=\mathcal{F}\mathbf{T}(k)=T.
\end{multline*}
Moreover, $T$ is connected, and $\mathbf{T}$ is a commutative group
scheme, so $T$ is abelian, hence nilpotent. Lemma~\ref{lem:Cartan-criterion}
now shows that $T$ is a Cartan subgroup of $G$.

Any Cartan subgroup $T'$ in $G$ is conjugate to $T$, that is, there
exists an element $g\in G$ such that, $T'=gTg^{-1}$. Then $\ad(g)(\mathbf{T})$
is a maximal torus in $\mathbf{G}$. Recall from Section~\ref{sec:Transp-norm}
that $\mathcal{F}$ preserves the conjugation action. Thus $\mathcal{F}(\ad(g))=\ad(g)$,
and we get $\mathcal{F}(\ad(g)(\mathbf{T}))=\ad(g)\mathcal{F}\mathbf{T}$.
Hence 
\[
\mathcal{F}(\ad(g)(\mathbf{T}))(k)=(\ad(g)(\mathcal{F}\mathbf{T}))(k)=g\mathcal{F}\mathbf{T}(k)g^{-1}=T'.
\]
This shows that the map $\mathbf{T}\mapsto T$ is surjective. To see
that it is injective, suppose that $\mathbf{T}$ and $\mathbf{S}$
are two maximal tori such that $T=\mathcal{F}\mathbf{T}(k)=\mathcal{F}\mathbf{S}(k)=S$.
Then $\mathbf{T}(A)=\mathbf{S}(A)$, and Proposition~\ref{pro:Nullstellensatz}
implies that $\mathbf{T}=\mathbf{S}$. $ $\end{proof}
\begin{rem*}
It may be tempting to try to prove that $T$ is a Cartan subgroup
of $G$ by showing directly that $T$ is the centraliser of a maximal
torus $T_{0}$ in $G$. By showing that the Greenberg functor preserves
centralisers, and using that $\C_{\mathbf{G}}(\mathbf{T})=\mathbf{T}$,
one could show that $\C_{G}(T)=T$. However, $T$ is not a maximal
torus of $G$ in general, or even of multiplicative type, so it does
not follow from this alone that $T$ is a Cartan subgroup. Moreover,
in general the maximal torus $T_{0}$ in $T$ will not be of the form
$(\mathcal{F}\mathbf{S})(k)$ for any torus $\mathbf{S}$ in $\mathbf{G}$
over $A$. This is the reason why we have proved Theorem~\ref{thm:Cartan}
using connected components of normaliser group schemes and the characterisation
of Cartan subgroups given by Lemma~\ref{lem:Cartan-criterion}.
\end{rem*}
Proposition~\ref{pro:Unip rad and B} and Theorem~\ref{thm:Cartan}
provide a vast generalisation of parts of a result of Hill (\cite{Hill_regular},
Proposition~2.2). 

As a consequence of Theorem~\ref{thm:Cartan}, show next how Proposition~\ref{pro:F and Tran}
can be strengthened in certain cases from a statement about equality
of $k$-points of schemes to equality of the schemes themselves (the
point being that the schemes in question are reduced over $k$).
\begin{prop}
Let $\mathbf{G}$ be a reductive group scheme over $A$. Let $\mathbf{H}$
and $\mathbf{H}'$ be smooth subgroup schemes of $\mathbf{G}$, and
let $\mathbf{T}$ and $\mathbf{T}'$ be maximal tori of $\mathbf{G}$.
Consider the action of $\mathbf{G}$ on itself by conjugation. Then
\[
\mathcal{F}\T_{\mathbf{G}}(\mathbf{T},\mathbf{H})=\T_{\mathcal{F}\mathbf{G}}(\mathcal{F}\mathbf{T},\mathcal{F}\mathbf{H}).
\]
Assume moreover that $ $$\mathbf{H}$ contains $\mathbf{T}$ and
$\mathbf{H}'$ contains $\mathbf{T}'$, respectively, and that $\mathbf{H}_{k}$
and $\mathbf{H}_{k}'$ are connected. Then 
\[
\mathcal{F}\T_{\mathbf{G}}(\mathbf{H},\mathbf{H}')=\T_{\mathcal{F}\mathbf{G}}(\mathcal{F}\mathbf{H},\mathcal{F}\mathbf{H}').
\]
\end{prop}
\begin{proof}
Since $\mathbf{G}$ is reductive, $\mathbf{T}$ is a Cartan subgroup
of $\mathbf{G}$, and by Theorem~\ref{thm:Cartan} $\mathcal{F}\mathbf{T}$
is a Cartan subgroup of $\mathcal{F}\mathbf{G}$. By \cite{SGA3},
XII~7.8, $\T_{\mathbf{G}}(\mathbf{T},\mathbf{H})$ is representable
by a closed smooth subscheme of $\mathbf{G}$, and $\T_{\mathcal{F}\mathbf{G}}(\mathcal{F}\mathbf{T},\mathcal{F}\mathbf{H})$
is representable by a closed smooth subscheme of $\mathcal{F}\mathbf{G}$.
Since $\T_{\mathcal{F}\mathbf{G}}(\mathcal{F}\mathbf{T},\mathcal{F}\mathbf{H})$
is smooth over an algebraically closed field, it is in particular
reduced. The assertion $\mathcal{F}\T_{\mathbf{G}}(\mathbf{T},\mathbf{H})=\T_{\mathcal{F}\mathbf{G}}(\mathcal{F}\mathbf{T},\mathcal{F}\mathbf{H})$
then follows from the equality of $k$-points given by Proposition~\ref{pro:F and Tran}.

Assume moreover that $\mathbf{H}$ contains $\mathbf{T}$, $\mathbf{H}'$
contains $\mathbf{T}'$, and that $\mathbf{H}_{k}$ and $\mathbf{H}_{k}'$
are connected. Then $\mathbf{H}$ and $\mathbf{H}'$ are subgroups
of type (R) of $\mathbf{G}$ (see \cite{SGA3}, XXII~5.2.1 for the
notion of subgroup of type (R)). Furthermore, by Theorem~\ref{thm:Cartan}
$\mathcal{F}\mathbf{T}$ and $\mathcal{F}\mathbf{T}'$ are Cartan
subgroups of $\mathcal{F}\mathbf{G}$, so $\mathcal{F}\mathbf{H}$
and $\mathcal{F}\mathbf{H}'$ are subgroups of type (R) of $\mathcal{F}\mathbf{G}$.
By \cite{SGA3}, XXII~5.3.9, $\T_{\mathbf{G}}(\mathbf{H},\mathbf{H}')$
is representable by a closed smooth subscheme of $\mathbf{G}$, and
$\T_{\mathcal{F}\mathbf{G}}(\mathcal{F}\mathbf{H},\mathcal{F}\mathbf{H}')$
is representable by a closed smooth subscheme of $\mathcal{F}\mathbf{G}$.
The assertion $\mathcal{F}\T_{\mathbf{G}}(\mathbf{H},\mathbf{H}')=\T_{\mathcal{F}\mathbf{G}}(\mathcal{F}\mathbf{H},\mathcal{F}\mathbf{H}')$
then follows from Proposition~\ref{pro:F and Tran} in the same way
as above.
\end{proof}
We conclude with some further results mentioned in the introduction.
\begin{prop}
\label{pro:N(P)-B}Let $\mathbf{G}$ be a reductive group scheme over
$A$, and let $\mathbf{P}$ be a parabolic subgroup of $\mathbf{G}$.
Then $\N_{G}(P)=P$. Moreover, let $\mathbf{B}$ and $\mathbf{B}'$
be two Borel subgroups of $\mathbf{G}$. Then $B$ and $B'$ are conjugate
in $G$. \end{prop}
\begin{proof}
From Corollary~\ref{cor:F - Norm} and the fact that $\N_{\mathbf{G}}(\mathbf{P})=\mathbf{P}$
(see Lemma~\ref{lem:Norm T - P}), we get
\[
\N_{G}(P)=\N_{\mathcal{F}\mathbf{G}(k)}(\mathcal{F}\mathbf{P}(k))=\N_{\mathcal{F}\mathbf{G}}(\mathcal{F}\mathbf{P})(k)=\mathcal{F}\N_{\mathbf{G}}(\mathbf{P})(k)=\mathcal{F}\mathbf{P}(k)=P.
\]
By \cite{SGA3}, XXII~5.3.9, the strict transporter $\T_{\mathbf{G}}(\mathbf{B},\mathbf{B}')$
is representable by a smooth scheme over $A$. Hence the reduction
map 
\[
\T_{\mathbf{G}}(\mathbf{B},\mathbf{B}')(A)\longrightarrow\T_{\mathbf{G}}(\mathbf{B},\mathbf{B}')(k)
\]
is surjective. Since the formation of transporters commutes with base
extension, we have 
\[
\T_{\mathbf{G}}(\mathbf{B},\mathbf{B}')(k)\cong\T_{\mathbf{G}}(\mathbf{B},\mathbf{B}')_{k}(k)\cong\T_{\mathbf{G}_{k}}(\mathbf{B}_{k},\mathbf{B}'_{k})(k)=\T_{\mathbf{G}_{k}(k)}(\mathbf{B}_{k}(k),\mathbf{B}'_{k}(k)).
\]
By definition, $\mathbf{B}_{k}(k)$ and $\mathbf{B}'_{k}(k)$ are
Borel subgroups in $\mathbf{G}{}_{k}(k)$, and it is well-known that
any two Borel subgroups of a linear algebraic group are conjugate.
Thus $\T_{\mathbf{G}}(\mathbf{B},\mathbf{B}')(k)$ is non-empty, and
so $\T_{\mathbf{G}}(\mathbf{B},\mathbf{B}')(A)$ is non-empty. By
Proposition~\ref{pro:F and Tran}, we have
\begin{multline*}
\T_{\mathbf{G}}(\mathbf{B},\mathbf{B}')(A)\\
=\mathcal{F}\T_{\mathbf{G}}(\mathbf{B},\mathbf{B}')(k)=\T_{\mathcal{F}\mathbf{G}}(\mathcal{F}\mathbf{B},\mathcal{F}\mathbf{B}')(k)=\T_{\mathcal{F}\mathbf{G}(k)}(\mathcal{F}\mathbf{B}(k),\mathcal{F}\mathbf{B}'(k))\\
=\T_{G}(B,B'),
\end{multline*}
and so $\T_{G}(B,B')$ is non-empty. Hence there exists an element
in $G$ that conjugates $B$ to $B'$.\end{proof}
\begin{acknowledgement*}
This work was supported at different times by EPSRC Grants GR/T21714/01
and EP/F044194/1, respectively. The author is grateful to T.~Ekedahl
for helpful discussions.
\end{acknowledgement*}
\bibliographystyle{alex}
\bibliography{alex}

\end{document}